\documentclass[11pt]{article}

\usepackage[mathscr]{eucal}
\usepackage{color}
\usepackage{amsmath}
\usepackage{amssymb}
\usepackage{amsfonts}
\usepackage{amscd}
\usepackage{amsthm}
\usepackage{latexsym}
\usepackage{graphicx}
\usepackage{subfig}
\usepackage{cleveref}
\usepackage[numbers]{natbib}

\def\be{\begin{equation}}
\def\ee{\end{equation}}
\def\bse{\begin{subequations}}
\def\ese{\end{subequations}}

\usepackage{latexsym}

\let\be\beta

\newtheorem{theorem}{Theorem}
\newtheorem{lemma}[theorem]{Lemma}
\newtheorem{proposition}[theorem]{Proposition}
\newtheorem{remark}[theorem]{Remark}

\newtheorem{definition}[theorem]{Definition}

\def\bse{\begin{subequations}}
\def\ese{\end{subequations}}

\usepackage{geometry}

\title{The Keller-Segel model with mass critical exponent}

\author{ {\bf\large Shen Bian}\thanks{Corresponding author. Email address: bianshen66@163.com}  \qquad  {\bf\large Yichen Zou}\thanks{ Email address: 2023201131@buct.edu.cn}  \\
\
\\
{\it\small Department of Mathematical Sciences, Beijing University of Chemical Technology} \\
{\it\small Beijing, 100029, P.R. CHINA }\\
\vspace{1.0mm} }
\date{}

\begin{document}
\let\cleardoublepage\clearpage%svinei kenes selides

\maketitle

\begin{abstract}
We consider a Keller-Segel model with non-linear porous medium type diffusion and non-local attractive power law interaction, focusing on potentials that are less singular than Newtonian interaction. Here, the nonlinear diffusion is chosen to be $m=2-\frac{2s}{d}$, in which case the steady states are compactly supported. We analyse under what conditions on the initial data the regime that attractive forces are stronger than diffusion occurs and classify the conditions for global existence and finite time blow-up of solutions. It is shown that there exists a threshold value which is characterized by the optimal constant of a variant of the Hardy-Littlewood-Sobolev inequality. Specifically, the solution will exist globally if the initial data is below the threshold, while the solution blows up in finite time when the initial data is above the threshold.
\end{abstract}
%%%%%%%%%%%%%%%%%%%%%%%%%%%%%%%%%%%%%%%%%%%%%%%%%%%%%%%%%%%%%%%%%%%%%%%%%%%%%%%%%%%%%%%%%%%%%%%%%%%%%%%%%%%%%%%%%%%%%%%%%%%%%%%%%%%%%%%%%%%%%%%%%%%%%%%%%%%%%%%%%%%%%%%%%%%%%%%%%%%%%%%%%%%%%%%%%%%%%%%%%%%%%%%%%%%%%%%%%%%%%%%%%%%%%%%%%%%%%%%%%%%%%%%%%%%%%%%%%%%%%%%%%%%%%%%%%%%%%%%%%%%%%%%%%%%%%%%%%%%%%%%%%%%%%%%%%%%%%%%%%%%%%%%%%%%%%%%%%%%%%%%%%%%%%%%%%%%%%%%%%%%%%%%%%%%%%%%%%%%%%%%%%%%%%%%%%%%%%%%%%%%%%%%%%%%%%%%%%%%%%%%%%%%%%%%%%%%%%%%%%%%%%%%%%%%%%%%%%%%%%%%%%%%%%%%%%%%%%%%%%%%%%%%%%%%%%%%%%%%%%%%%%%%%%%%%%%%%%%%%%%%%%%%%%%%%%%%%%%%%%%%%%%%%%%%%%%%%%%%%%%%%%%%%%%%%%%%%%%%%%%%%%%%%%%%%%%%%%%%%%%%%%%%%%%%%%%%%%%%%%%%%%%%%%%%%%%%%%%%%%%%%%%%%%%%%%%%%%%%%%%%%%%%%%%%%%%%%%%%%%%%%%%%%%%%%%%%%%%%%%%%%%%%%%%%%%%%%%%%%%%%%%%%%%%%%%%%%%%%%%%%%%%%%%%%%%

\section{Introduction}

In this paper, we investigate qualitative properties of non -negative solutions for the aggregation-diffusion equation
\begin{align}\label{PKS}
\left\{\begin{array}{l}
u_{t}=\Delta u^{m}-\nabla \cdot(u \nabla \phi), \\
\left(-\Delta\right)^{s} \phi(x, t)=u(x, t), \\
u(x, 0)=u_{0}(x) \geq 0,
\end{array}\right.
\end{align}
where the diffusion exponent is taken to be $m>1$, $2<2s<d$. $u$ represents the density of cells, the chemoattractant $\phi$ is governed by a diffusion process and can be expressed by the convolution of Riesz potential of $u(x, t)$, that is
\begin{align}\label{phi}
\phi\left(x, t\right)=c_{d,s} \int_{{R}^{d}} \frac{u(y, t)}{|x-y|^{d-2 s}} d y,
\end{align}
where $c_{d,s}$ is given by
\begin{align}\label{cds}
c_{d,s}=\frac{\Gamma(d/2-s)}{\pi^{d/2}4^s\Gamma(s)}.
\end{align}
The initial data throughout this paper is assumed to satisfy
\begin{align}\label{initial}
u_0\in L_+^1\cap L^{\infty}(\mathbb{R}^d),\quad \int_{\mathbb{R}^d}|x^2|u_0(x)dx<\infty.
\end{align}

A fundamental property of the solutions to \eqref{PKS} is the formal conservation of the total mass of the system
\begin{align}\label{Mass}
M:=\int_{\mathbb{R}^d}u_0(x)dx=\int_{\mathbb{R}^d}u(x)dx,\quad \text{for}~~ t\geq 0.
\end{align}

Nonlinear diffusion with $m > 1$ describes the local repulsion of cells while the attractive nonlocal term models cell movement toward chemotactic sources or attractive interaction between cells due to cell adhesion. The main characteristic of equation \eqref{PKS} is the competition between the diffusion and the nonlocal aggregation. This is well presented by the free energy
\begin{align}\label{free energy}
F\left(u\right)&=\frac{1}{m-1} \int_{{R}^{d}} u^{m} d x-\frac{1}{2} \int_{{R}^{d}} \phi u d x \nonumber\\
&=\frac{1}{m-1} \int_{{R}^{d}} u^{m} d x-\frac{1}{2}c_{d,s} \iint_{{R}^{d} \times {{R}^{d}}} \frac{u(x, t) u(y, t)}{|x-y|^{d-2 s}} d x d y .
\end{align}
The competition between these two terms leads to finite time blow-up and global existence. As a matter of fact, \eqref{PKS} can be recast as
\begin{align}\label{free1}
    u_t=\nabla\cdot\left(u\nabla\mu\right),
\end{align}
where $\mu$ is the chemical potential
\begin{align}\label{free2}
    \mu=\frac{m}{m-1} u^{m-1}- \phi.
\end{align}
Multiplying \eqref{free1} by $\mu$ and integrating it in space one obtains
\begin{align}\label{free3}
\frac{dF(u)}{dt}=-\int_{\mathbb{R}^d} u|\nabla\mu|^2 dx\leq 0.
\end{align}
This implies that $F[u(\cdot, t)]$ is non-increasing with respect to time.

When $s=1$, \eqref{PKS} becomes the classical Keller-Segel model \cite{KS70}, which has been extensively explored for \cite{B13,A09,A08,A06,V12,J19,J15,I12,P07,S11,S06,Y06}. When $s\neq 1$, spatial fractional derivatives can be used for anomalous diffusion or dispersion when a particle plume spreads at a rate inconsistent with the Brownian motion models. The appearance of spatial fractional derivatives in diffusion equations are exploited for macroscopic description of transport and often lead to super-diffusion phenomenon \cite{L17}. The fractional Laplacian operators $\left(-\Delta\right)^s$ for $0<s<1$ models anomalous diffusion and enables the study of long-range diffusive interactions \cite{S04,M09,AM11}, and the hyper-viscosity $(s>1)$ is employed in turbulence modeling \cite{U08}.

Notice that equation \eqref{PKS} possesses a scaling invariance which leaves the $L^p$-norm invariant and produces a balance between the diffusion and the aggregation terms where
\begin{align*}
    p=\frac{d(2-m)}{2s}.
\end{align*}

Indeed, if $u(x, t)$ is a solution to \eqref{PKS}, then $u_{\lambda}(x, t)=\lambda u\left(\lambda^{\frac{2-m}{2 s}} x, \lambda^{m-1+\frac{2-m}{s}} t\right)$ is also a solution to \eqref{PKS} and this scaling preserves the $L^{p}$-norm $\left\|u_{\lambda}\right\|_{L^{p}(\mathbb{R}^{d})}=\|u\|_{L^{p}(\mathbb{R}^{d})}$. For the critical case  $m_{*}=2-2 s / d$, the above scaling becomes the mass invariant scaling $u_{\lambda}(x, t)=\lambda u\left(\lambda^{1 / d} x, \lambda^{1+\frac{2-2 s}{d}} t\right)$. For the subcritical case $m>2-2 s / d$, the aggregation dominates the diffusion for low density and prevents spreading. While for high density, the diffusion dominates the aggregation thus blow-up is precluded. On the contrary, for the supercritical case $1<m<2-2 s / d$, the aggregation dominates the diffusion for high density (large $\lambda$) and the density may have finite time blow-up. While for low density (small $\lambda$), the diffusion dominates the aggregation and the density has infinite time spreading. These behaviors also appear in many other physical systems such as thin film, Hele-Shaw, stellar collapse as well as the nonlinear Schrödinger equation \cite{T04}.

When $s\neq 1$, we stress that a lot of efforts have been put to investigate the steady states of \eqref{PKS} in terms of the diffusion exponent $m$ \cite{V21,V17,J19,J18,H20}. Specifically, for $m>2-2 s / d$, this is the diffusion-dominated regime. There exists a unique up to translations compactly supported steady state to \eqref{PKS}. Moreover, such steady state is radially decreasing and it is the global minimiser of the energy functional $F(u)$. We refer to \cite{J18} for details. For $1<m<2-2 s / d$, this is the aggregation-dominated regime. There are two subcases for stationary states \cite{B24}. For $\frac{2 d}{d+2 s}<m<2-\frac{2 s}{d}$, all the stationary states are compactly supported. While they are not compactly supported for $1<m \leq \frac{2 d}{d+2 s}$ and there is an explicit formula of one parameter family of stationary states in the case $m=\frac{2 d}{d+2 s}$. For $m=2-\frac{2s}{d}$, the existence of the stationary states can only happen for a critical mass $M_*$ which is characterised by the optimal
constant of a variant of the HLS inequality.

In this work, we mainly focus on dynamic solution to \eqref{PKS}. It should be pointed out that in the case $0<s<1$, by observing $2s-2<0$ in \eqref{est2} see \Cref{local}, the dynamic solution will exist globally for any initial data. When $s>1$, for $m>2-2 s/d$, the solution exists globally for any initial data \cite{J19}. For $1<m<2-2 s/d$, it has been prove in \cite{B24} that there is a critical threshold on the initial data that separates global case and blow-up of finite time solutions to \eqref{PKS}. \textbf{Our main goal} in this paper is to detect dynamic behaviors of solutions to \eqref{PKS} for the case $m=2-2 s/d$ with $s>1$.

Before proceeding further to show the dynamic behavior of the solution, we define the weak solution which we will deal with in this paper.

\begin{definition}[Weak solution]Let $u_0$ be an initial condition satisfying \eqref{initial} and $T\in(0,\infty]$. A weak solution to \eqref{PKS} with initial data $u_0$ is a non-negative function $u\in L^\infty(0,T;L_+^1\cap L^\infty(\mathbb{R}^d))$ and for any $0<t<T$
\begin{align}\label{weak}
    \int_{\mathbb{R}^d}\psi u(\cdot,t)-\int_{\mathbb{R}^d}\psi u_0(x)=&\int_0^t \int_{\mathbb{R}^d}\Delta\psi u^mdxds\nonumber\\
    &-\frac{1}{2}\int_0^t \iint\limits_{\mathbb{R}^d\times \mathbb{R}^d}\frac{[\nabla\psi(x)-\nabla\psi(y)]\cdot(x-y)}{|x-y|^2}\frac{u(x,s)u(y,s)}{|x-y|^{d-2s}}
  \end{align}
for $\forall \psi\in C_0^\infty (\mathbb{R}^d).$
\end{definition}

Here, we will make a fundamental use of a variant to the HLS(VHLS) inequality, see \Cref{VHLS}: for all $h \in L^{1}(\mathbb{R}^{d}) \cap {L}^{m}(\mathbb{R}^{d})$, there exists an optimal constant $C_{*}$ such that
\begin{align}\label{1HLS1}
\left|\iint_{\mathbb{R}^{d} \times \mathbb{R}^{d}} \frac{h(x) h(y)}{|x-y|^{d-2}} dx dy\right| \leq C_{*}\|h\|_{L^{m}(\mathbb{R}^{d})}^{m}\|h\|_{L^{1}(\mathbb{R}^{d})}^{2s / d} .
\end{align}
The VHLS inequality and the identification of the equality cases allow us to give \textbf{the main result} of this work, namely, the following
sharp critical mass
\begin{align}\label{1HLS2}
M_{*}=\left[\frac{2}{(m-1)C_{*}c_{d,s}}\right]^{\frac{d}{2s}}
\end{align}
for \eqref{PKS}. More precisely, we will show that global solutions exist for $M > M_*$, see \Cref{GE}, while there are finite time blowing-up solutions otherwise for $M < M_*$, see \Cref{B}.

The results are organised as follows. This work is entirely devoted to establishing an exact criteria on the initial data for the global existence and finite time blow-up of solutions to \eqref{PKS}. With this aim we firstly construct the existence criterion which shows a key maximal existence time for the free energy solution to \eqref{PKS} in Section 2. Section 3 mainly explores the variational structure of the modified Hardy-Littlewood-Sobolev inequality. Finally, Section 4 makes use of the extremum function of the VHLS inequality to prove the main theorem concerning the dichotomy. The global existence for small initial data and finite time blow-up for large initial data.

%%%%%%%%%%%%%%%%%%%%%%%%%%%%%%%%%%%%%%%%%%%%%%%%%%%%%%%%%%%%%%%%%%%%%%%%%%%%%%%%%%%%%%%%%%%%%%%%%%%%%%%%%%%%%%%%%%%%%%%%%%%%%%%%%%%%%%%%%%%%%%%%%%%%%%%%%%%%%%%%%%%%%%%%%%%%%%%%%%%%%%%%%%%%%%%%%%%%%%%%%%%%%%%%%%%%%%%%%%%%%%%%%%%%%%%%%%%%%%%%%%%%%%%%%%%%%%%%%%%%%%%%%%%%%%%%%%%%%%%%%%%%%%%%%%%%%%%%%%%%%%%%%%%%%%%%%%%%%%%%%%%%%%%%%%%%%%%%%%%%%%%%%%%%%%%%%%%%%%%%%%%%%%%%%%%%%%%%%%%%%%%%%%%%%%%%%%%%%%%%%%%%%%%%%%%%%%%%%%%%%%%%%%%%%%%%%%%%%%%%%%%%%%%%%%%%%%%%%%%%%%%%%%%%%%%%%%%%%%%%%%%%%%%%%%%%%%%%%%%%%%%%%%%%%%%%%%%%%%%%%%%%%%%%%%%%%%%%%%%%%%%%%%%%%%%%%%%%%%%%%%%%%%%%%%%%%%%%%%%%%%%%%%%%%%%%%%%%%%%%%%%%%%%%%%%%%%%%%%%%%%%%%%%%%%%%%%%%%%%%%%%%%%%%%%%%%%%%%%%%%%%%%%%%%%%%%%%%%%%%%%%%%%%%%%%%%%%%%%%%%%%%%%%%%%%%%%%%%%%%%%%%%%%%%%%%%%%%%%%%%%%%%%%%%%%%%%%%%%%%%%%%%%%%%%%%%%%%%%%%%%%%%%%%%%%%%%%%%%%%%%%%%%%%%%%%%%%%%%%%%%%%%%%%%%%%%%%%%%%%%%%%%%%%%%%%%%%%%%%%%%%%%%%%%%%%%%%%%%%%%%%%%%%%%%%%%%%%%%%%%%%%%%%%%%%%%%%%%%%%%%%%%%%%%%%%%%%%%%%%%%%%%%%%%%%%%%%%%%%%%%%%%%%%%%%%%%%%%%%%%%%%%%%%%%%%%%%%%%%%%%%%%%%%%%%%%%%%%%%%%%%%%%%%%%%%%%%%%%%%%%%%%%%%%%%%%%%%%%%%%%%%%%%%%%%%%%%%%%%%%%%%%%%%%%%%%%%%%%%%%%%%%%%%%%%%%%%%%%%%%%%%%%%%%%%%%%%%%%%%%%%%%%%%%%%%%%%%%%%%%%%%%%%%%%%%%%%%%%%%%%%%%%%%%%%%%%%%%%%%%%%%%%%%%%%%%%%%%%%%%%%%%%%%%%%%%%%%%%%%%%%%%%%%%%%%%%

\section{Local existence and blow-up criteria} \label{Local existence}

Consider regularization issues
\begin{align}\label{regularization1}
\left\{\begin{array}{l}
\partial_{t} u_{\varepsilon}=\Delta u_{\varepsilon}^{m}+\varepsilon \Delta u_{\varepsilon}-\nabla \cdot\left(u_{\varepsilon} \nabla \phi_{\varepsilon}\right), ~x \in {R}^{d},~ t>0, \\
u_{\varepsilon}(x, 0)=u_{0 \varepsilon} \geq 0,
\end{array}
\right.
\end{align}
where $\phi_{\varepsilon}$ is given by
\begin{align}\label{regularization2}
\phi_{\varepsilon}=R_{\varepsilon} * u_{\varepsilon},
\end{align}
with the regularized Riesz potential
\begin{align}\label{regularization3}
R_{\varepsilon}(x)=\frac{c_{d,s}}{\left(|x|^{2}+\varepsilon^{2}\right)^{\frac{d-2 s}{2}}},\nonumber
\end{align}
where $c_{d,s}=\frac{\Gamma(d/2-s)}{\pi^{d/2}4^s\Gamma(s)}$. Here $u_{0 \varepsilon}$ is a sequence of approximation for $u_{0}$ and can be constructed to satisfy that there exists $\varepsilon_{0}>0$ such that for any $0<\varepsilon<\varepsilon_{0}$,
\begin{align}
\left\{\begin{array}{l}
u_{0 \varepsilon} \geq 0,\left\|u_{\varepsilon}(x, 0)\right\|_{L^{1}\left(\mathbb{R}^{d}\right)}=\left\|u_{0}\right\|_{L^{1}\left(\mathbb{R}^{d}\right)}, \\
\int_{\mathbb{R}^{d}}|x|^{2} u_{0 \varepsilon} d x \rightarrow \int_{\mathbb{R}^{d}}|x|^{2} u_{0}(x) d x, \text { as } \varepsilon \rightarrow 0, \\
u_{0 \varepsilon}(x) \rightarrow u_{0}(x) \text { in } L^{q}\left(\mathbb{R}^{d}\right), \text { for } 1 \leq q<\infty, \text { as } \varepsilon \rightarrow 0 .
\end{array}\right.
\end{align}
This regularized problem has global in time smooth solutions for any  $\varepsilon>0$. This approximation has been proved to be convergent. More precisely, following the arguments in \citep[Theorem 4.2]{B14}, \citep[Lemma 4.8]{J19} and \citep[Section 4]{Y06} we assert that if
\begin{align}\label{regularization5}
\left\|u_{\varepsilon}(\cdot, t)\right\|_{L^{\infty}(\mathbb{R}^{d})}<C_{0},
\end{align}
where $C_{0}$ is independent of $\varepsilon>0$, then there exists a subsequence $\varepsilon_{n} \rightarrow 0$ such that
\begin{align}\label{regularization6}
u_{\varepsilon_{n}} \rightarrow u \text { in } L^{r}\left(0, T ; L^{r}\left(\mathbb{R}^{d}\right)\right), 1 \leq r<\infty
\end{align}
and $u$ is a weak solution to \eqref{PKS} on $[0, T)$.

According to the above analysis, a weak solution to \eqref{PKS} on $[0, T)$ exists when \eqref{regularization5} is fulfilled. So we shall focus on establishing the availability of the $L^\infty$-bound. As we will see in the following theorem where the local in time existence and blow-up criteria are constructed, such a bound follows from the $L^r$-norm for $r > 1$ which additionally provides a characterisation of the maximal existence time.

Before showing the local in time existence, we prepare the following HLS inequality which is stated in \cite{E01}.

\begin{lemma}[HLS inequality]\label{HLS}
    Let $q, q_{1}>1$, $d \geq 3$ and $0<\beta<d$ with $1 / q+1 / q_{1}+ \beta / d=2$. Assume $f \in L^{q}(\mathbb{R}^{d})$ and $g \in L^{q_{1}}(\mathbb{R}^{d})$. Then there exists a sharp constant $C_{HLS}$ independent of $f$ and $g$ such that
\begin{align}\label{HLS1}
\left|\iint_{\mathbb{R}^{d} \times \mathbb{R}^{d}} \frac{f(x) g(y)}{|x-y|^{\beta}} d x d y\right| \leq C(d,s,q)\|f\|_{L^{q}(\mathbb{R}^{d})}\|g\|_{L^{q_{1}}(\mathbb{R}^{d})} .
\end{align}
If $q=q_{1}=\frac{2 d}{2 d-\beta}$, then
\begin{align}\label{HLS2}
C(d,s,q)=C(d,s)=\pi^{\beta / 2} \frac{\Gamma(d / 2-\beta / 2)}{\Gamma(d-\beta / 2)}\left(\frac{\Gamma(d / 2)}{\Gamma(d)}\right)^{-1+\beta / d} .
\end{align}
In this case there is equality in \eqref{HLS1} if and only if $h \equiv(  const.)  f$ and
\begin{align}\label{HLS3}
f(x)=A\left(\gamma^{2}+\left|x-x_{0}\right|^{2}\right)^{-(2 d-\beta) / 2}
\end{align}
for some $A>0$, $\gamma \in \mathbb{R}$ and $x_{0} \in \mathbb{R}^{d}$. Particularly, for $\beta=d-2 s$,
\begin{align}\label{HLS4}
f=U_{s}(\lambda)\quad \text { for some } \lambda>0 \text {. }
\end{align}
\end{lemma}

\begin{theorem}[Local in time existence and blow-up criteria]\label{local}
Under assumption \eqref{initial} on the initial condition, there are a maximal existence time $T_{w} \in(0, \infty]$ and a weak solution $u$ to \eqref{PKS} on $\left[0, T_{w}\right)$.
\end{theorem}
\begin{proof}
To prove this result we need to refine the arguments used in the proof of \citep[Theorem 2.11]{B13}. We follow a procedure analogous to the ones in \citep[Lemma 2.3]{A09}. It’s obtained by multiplying equation \eqref{regularization1} with $ru^{r-1}$ that
\begin{align}\label{est1}
& \frac{d}{d t} \int_{{R}^{d}} u_{\varepsilon}^{r} d x+\frac{4 m r(r-1)}{(m+r-1)^{2}} \int_{{R}^{d}}\left|\nabla u_{\varepsilon}^{\frac{m+r-1}{2}}\right|^{2} d x+\varepsilon \frac{4(r-1)}{r} \int_{{R}^{d}}\left|\nabla u_{\varepsilon}^{r / 2}\right|^{2} d x \nonumber\\
= & (2 s-2)(r-1) \iint_{{R}^{d} \times {R}^{d}} \frac{u_{\varepsilon}^{r}(x) u_{\varepsilon}(y)}{\left(|x-y|^{2}+\varepsilon^{2}\right)^{\frac{d+2-2 s}{2}}} d x d y.
\end{align}
Due to $2s > 2$ and $r>1$, we treat the right hand side of \eqref{est1} by applying the HLS inequality \eqref{HLS1} with $f=u_\varepsilon^r(x)$ and $g=u_\varepsilon(y)$ such that
\begin{align}\label{est2}
\iint_{\mathbb{R}^{d} \times \mathbb{R}^{d}} \frac{u_{\varepsilon}^{r}(x) u_{\varepsilon}(y)}{\left(|x-y|^{2}+\varepsilon^{2}\right)^{\frac{d+2-2 s}{2}}} d x d y & \leq \iint_{\mathbb{R}^{d} \times \mathbb{R}^{d}} \frac{u_{\varepsilon}^{r}(x) u_{\varepsilon}(y)}{|x-y|^{d+2-2 s}} d x d y \nonumber\\
& \leq C_{HLS}\left\|u_{\varepsilon}^{r}\right\|_{L^{q}(\mathbb{R}^{d})}\left\|u_{\varepsilon}\right\|_{L^{q_{1}}(\mathbb{R}^{d})} \nonumber\\
& =C_{HLS}\left\|u_{\varepsilon}\right\|_{L^{\frac{r+1}{1+(2 s-2) / d}}(\mathbb{R}^{d})}^{r+1},
\end{align}
where $\frac{1}{q}+\frac{1}{q_1}+\frac{d+2-2s}{d}=2$ and we have used $q_1=qr=\frac{d(r+1)}{d+2s-2}$, $C_{HLS}$ is a bounded constant depending on $d,~s,~r$. Furthermore, for $1 \leq k_1<k_2<k_3=\frac{2 d}{d-2}$ and $\frac{1}{k_2}=\frac{\theta}{k_1}+\frac{1-\theta}{k_3}$,
we have the following GNS inequality \cite{E01}: there exists a positive constant $C_{GNS}$ such that
\begin{align}\label{est3}
\|v\|_{L^{k_2}(\mathbb{R}^d)}\leq C_{GNS}\|v\|_{L^{k_1}(\mathbb{R}^d)}^{\theta}\|\nabla v\|_{L^2(\mathbb{R}^d)}^{1-\theta},\quad \theta=\left(\frac{1}{k_2}-\frac{1}{k_3}\right)\left(\frac{1}{k_1}-\frac{1}{k_3}\right)^{-1},
\end{align}
which we apply with $v=u_{\varepsilon}^{\frac{m+r-1}{2}}$ to discover
\begin{align}\label{est4}
\|u_{\varepsilon}\|_{L^{\frac{r+1}{1+(2s-2)/d}}(\mathbb{R}^d)}^{r+1}\leq C_{GNS}\|u_{\varepsilon}\|_{L^{k_1\frac{m+r-1}{2}}(\mathbb{R}^d)}^{a\theta\frac{m+r-1}{2}}\|\nabla u_{\varepsilon}^{\frac{m+r-1}{2}}\|_{L^2(\mathbb{R}^d)}^{a(1-\theta)},
\end{align}
where the parameters satisfy
\begin{align}\label{est5}
 a\frac{m+r-1}{2}=r+1,\quad k_2\frac{m+r-1}{2}=\frac{r+1}{1+(2s-2)/d}.
\end{align}
Here we pick
\begin{align}\label{est6}
1=\frac{d(2-m)}{2s}<k_1\frac{m+r-1}{2}=r.
\end{align}
Then simple computations show that
\begin{align}\label{est7}
a(1-\theta)=2+\frac{\frac{-2m+4}{m+r-1}\frac{1}{k_1}-\frac{2s}{d}}{\left(\frac{1}{k_1}-\frac{d-2}{2d}\right)}<2
\end{align}
in case of $1 < m = 2- 2s/d$. By Young’s inequality, from \eqref{est4} we proceed to obtain
\begin{align}\label{est8}
  \|u_{\varepsilon}\|_{L^{\frac{r+1}{1+(2s-2)/d}}(\mathbb{R}^d)}^{r+1} \leq C_1\|\nabla u_{\varepsilon}^{\frac{m+r-1}{2}}\|_{L^2(\mathbb{R}^d)}^2+C_1^{-p/q}\left(\|u_{\varepsilon}\|_{L^{r}(\mathbb{R}^d)}^{a\theta\frac{m+r-1}{2}}\right)^{\frac{2}{2-a(1-\theta)}},
\end{align}
where the constant $C_1$ is defined as in \eqref{HLS2}. \eqref{est8} can be read as
\begin{align}\label{est9}
(2 s-2) C_{HLS}(r-1)\left\|u_{\varepsilon}\right\|_{L^{\frac{r+1}{1+(2 s-2) / d} }(\mathbb{R}^{d})}^{r+1} \leq & \frac{2 m r(r-1)}{(m+r-1)^{2}}\left\|\nabla u_{\varepsilon}^{\frac{m+r-1}{2}}\right\|_{L^{2}(\mathbb{R}^{d})}^{2} \nonumber\\
& +(2 s-2) C_{HLS}(r-1)C_1^{-p/q}\left(\left\|u_{\varepsilon}\right\|_{L^{r}(\mathbb{R}^{d})}^{r}\right)^{1+\frac{\frac{2 s}{d}-(2-m) \frac{2 s-2}{d}}{\frac{2 s}{d} r-(2-m)}}
\end{align}
for $r>\frac{d(2-m)}{2 s}=1$ due to $m=2-\frac{2 s}{d}$. Therefore, substituting \eqref{est2} and \eqref{est9} into \eqref{est1} we thus end up with
\begin{align}\label{est10}
& \frac{d}{d t}\left\|u_{\varepsilon}\right\|_{L^{r}(\mathbb{R}^{d})}^{r}+\frac{2 m r(r-1)}{(m+r-1)^{2}} \int_{\mathbb{R}^{d}}\left|\nabla u_{\varepsilon}^{\frac{m+r-1}{2}}\right|^{2} d x+\varepsilon \frac{4(r-1)}{r} \int_{\mathbb{R}^{d}}\left|\nabla u_{\varepsilon}^{r / 2}\right|^{2} d x \\
\leq&C\left(\left\|u_{\varepsilon}\right\|_{L^{r}(\mathbb{R}^{d})}^{r}\right)^{1+\frac{\frac{2 s}{d}-(2-m)\frac{2 s-2}{d}}{\frac{2 s}{d} r-(2-m)}}.
\end{align}
Hence the local in time $L^{r}$ -estimates are followed
\begin{align}\label{est22}
\left\|u_{\varepsilon}(\cdot, t)\right\|_{L^{r}(\mathbb{R}^{d})}^{r} & \leq \frac{1}{\left(\left\|u_{\varepsilon}(0)\right\|_{L^{r}(\mathbb{R}^{d})}^{-r \delta}-\delta C t\right)^{\frac{1}{\delta}}} \nonumber\\
& =\left(\frac{C(\delta)}{T_{0}-t}\right)^{\frac{1}{\delta}}, \quad T_{0}=\frac{\left\|u_{\varepsilon}(0)\right\|_{L^{r}(\mathbb{R}^{d})}^{-r \delta}}{\delta C},
\end{align}
where $\delta=\frac{1-\frac{(2-m)(2 s-2)}{2 s}}{r-\frac{d(2-m)}{2 s}}>0$ for $r>\frac{d(2-m)}{2 s}=1$ because of $m=2-\frac{2 s}{d}$. Moreover, coming back to \eqref{regularization6}, we further deduce that
\begin{align}\label{est21}
\left\|\nabla u_{\varepsilon}^{\frac{m+r-1}{2}}\right\|_{L^{2}(0, ~T_{0} ;~ L^{2}(\mathbb{R}^{d}))} \leq C\left(\left\|u_{\varepsilon}(0)\right\|_{L^{r}(\mathbb{R}^{d})}\right) .
\end{align}
Particularly, taking $r=m$ in \eqref{est22} leads to
\begin{align}\label{est23}
\left\|u_{\varepsilon}(\cdot, t)\right\|_{L^{m}(\mathbb{R}^{d})} =\left(\frac{C(\delta)}{T_{0}-t}\right)^{\frac{1}{m\delta}}, \quad T_{0}=\frac{\left\|u_{\varepsilon}(0)\right\|_{L^{r}(\mathbb{R}^{d})}^{-r \delta}}{\delta C}.
\end{align}
Thus following similar arguments of \citep[Section 4]{Y06} with the aid of regularities \eqref{est21} and \eqref{est23}, we deduce that there exists a weak solution $u$ of \eqref{PKS} by passing to the limit $u_\varepsilon \to u$ as $\varepsilon\to0$ (without relabeling).
Finally, a direct consequence of \citep[Theorem 2.4]{A09} characterises the maximal existence time of the weak solution and completes the proof.
\end{proof}

%%%%%%%%%%%%%%%%%%%%%%%%%%%%%%%%%%%%%%%%%%%%%%%%%%%%%%%%%%%%%%%%%%%%%%%%%%%%%%%%%%%%%%%%%%%%%%%%%%%%%%%%%%%%%%%%%%%%%%%%%%%%%%%%%%%%%%%%%%%%%%%%%%%%%%%%%%%%%%%%%%%%%%%%%%%%%%%%%%%%%%%%%%%%%%%%%%%%%%%%%%%%%%%%%%%%%%%%%%%%%%%%%%%%%%%%%%%%%%%%%%%%%%%%%%%%%%%%%%%%%%%%%%%%%%%%%%%%%%%%%%%%%%%%%%%%%%%%%%%%%%%%%%%%%%%%%%%%%%%%%%%%%%%%%%%%%%%%%%%%%%%%%%%%%%%%%%%%%%%%%%%%%%%%%%%%%%%%%%%%%%%%%%%%%%%%%%%%%%%%%%%%%%%%%%%%%%%%%%%%%%%%%%%%%%%%%%%%%%%%%%%%%%%%%%%%%%%%%%%%%%%%%%%%%%%%%%%%%%%%%%%%%%%%%%%%%%%%%%%%%%%%%%%%%%%%%%%%%%%%%%%%%%%%%%%%%%%%%%%%%%%%%%%%%%%%%%%%%%%%%%%%%%%%%%%%%%%%%%%%%%%%%%%%%%%%%%%%%%%%%%%%%%%%%%%%%%%%%%%%%%%%%%%%%%%%%%%%%%%%%%%%%%%%%%%%%%%%%%%%%%%%%%%%%%%%%%%%%%%%%%%%%%%%%%%%%%%%%%%%%%%%%%%%%%%%%%%%%%%%%%%%%%%%%%%%%%%%%%%%%%%%%%%%%%%%%%%%%%%%%%%%%%%%%%%%%%%%%%%%%%%%%%%%%%%%%%%%%%%%%%%%%%%%%%%%%%%%%%%%%%%%%%%%%%%%%%%%%%%%%%%%%%%%%%%%%%%%%%%%%%%%%%%%%%%%%%%%%%%%%%%%%%%%%%%%%%%%%%%%%%%%%%%%%%%%%%%%%%%%%%%%%%%%%%%%%%%%%%%%%%%%%%%%%%%%%%%%%%%%%%%%%%%%%%%%%%%%%%%%%%%%%%%%%%%%%%%%%%%%%%%%%%%%%%%%%%%%%%%%%%%%%%%%%%%%%%%%%%%%%%%%%%%%%%%%%%%%%%%%%%%%%%%%%%%%%%%%%%%%%%%%%%%%%%%%%%%%%%%%%%%%%%%%%%%%%%%%%%%%%%%%%%%%%%%%%%%%%%%%%%%%%%%%%%%%%%%%%%%%%%%%%%%%%%%%%%%%%%%%%%%%%%%%%%%%%%%%%%%%%%%%%%%%%%%%%%%%%%%%%%%%%%%%%%%%%%%%%%%%%%%%%%%%%%%%%%

\section{A variant of the HLS inequality} \label{globalfinite}

As we have just seen in the existence proof, the existence time of a weak solution to \eqref{PKS} heavily depends on the behaviour of its $L^m$-norm. As the free energy $F$ involves the $L^m$-norm, the information given by $F$ will be of paramount importance. Let us then proceed to a deeper study of this functional.

In this subsection, we will show the existence of the extremal function for the modified HLS inequality which is stated as \Cref{HLS}. To this end, we first define
\begin{align}\label{omega}
\omega(u):=\iint_{R^{d} \times R^{d}} \frac{u(x) u(y)}{|x-y|^{d-2s}} d x d y ,
\end{align}
and
\begin{align}\label{J}
J(u):=\frac{\omega(u)}{\|u\|_{L^1(\mathbb{R}^{d})}^{2s / d}\|u\|_{L^m(\mathbb{R}^{d})}^{m}},\quad u\in {L}^{1}(\mathbb{R}^{d}) \cap {L}^{m}(\mathbb{R}^{d}).
\end{align}
We now sketch to establish a variant of the HLS(VHLS) inequality.
\begin{proposition}[VHLS inequality]\label{VHLS}
Let $m=2-2s/d$. Let $u\in L^{1}(\mathbb{R}^{d}) \cap L^{m}(\mathbb{R}^{d})$. Then there exists an optimal constant $C_*$ such that
\begin{align}\label{VHLS1}
J(u)\leq C_*.
\end{align}
The optimal constant satisfies
\begin{align}\label{VHLS2}
C_*=\pi^{(d-2s) / 2} \frac{\Gamma(s)}{\Gamma((d+2s) / 2)}\left(\frac{\Gamma(d / 2)}{\Gamma(d)}\right)^{-2s / d} .
\end{align}
The equality holds in \eqref{VHLS1} if $u=\lambda U(\mu |x-x_0|)$ for any real numbers $\lambda>0$, $\mu>0$ and $x_0\in \mathbb{R}^d$. $U$ is a radially symmetric, non-increasing and compactly supported function.
\end{proposition}
\begin{proof}
{\it\textbf{Step 1}} (Existence of $C_*$)\quad
 Consider $u \in {L}^{1}(\mathbb{R}^{d}) \cap {L}^{m}(\mathbb{R}^{d})$. Applying the HLS inequality \eqref{HLS1} with  $q=q_1=2 d /(d+2s)$  and  $\beta=d-2s$, and then the H\"{o}lder inequality with $1<   2 d /(d+2s)<2-2s/d=m$, we obtain
\begin{align}\label{VHLS3}
\left|\iint_{{R}^{d} \times {{R}^{d}}} \frac{u(x, t) u(y, t)}{|x-y|^{d-2 s}} d x d y\right|\leq C(d,s)\|u\|_{L^{\frac{2d}{d+2s}}(\mathbb{R}^d)}^2
\leq C(d,s)\|u\|_{L^1(\mathbb{R}^d)}^{\frac{2s}{d}}\|u\|_{L^m(\mathbb{R}^d)}^m,
\end{align}
where
\begin{align}\label{VHLS4}
C(d,s)=\pi^{(d-2s) / 2} \frac{\Gamma(s)}{\Gamma((d+2s) / 2)}\left(\frac{\Gamma(d / 2)}{\Gamma(d)}\right)^{-2s / d} .
\end{align}
Consequently, $C_*$ is finite and bounded from above by $C(d,s)$.

{\it\textbf{Step 2}} (Existence of maximisers)\quad
Consider a maximising sequence $(u_j)_j$ in ${L}^{1}(\mathbb{R}^{d}) \cap \mathrm{L}^{m}(\mathbb{R}^{d})$, that is
\begin{align}\label{C*2}
\lim\limits_{j\to \infty}J(u_j)=C_*\text{.}
\end{align}
    We first prove that we may assume that $u_j$ is a non-negative, radially symmetric, non-increasing function such that $\|u_j\|_{L^1(\mathbb{R}^d)}=\|u_j\|_{L^m(\mathbb{R}^d)}=1$ for any $j\geq 0$. Indeed $J(u_j)\leq J(|u_j|)$ so that $(|u_j|)_j$ is also a maximising sequence. Next, denoting by $u_j^{*}$ the symmetric decreasing rearrangement of $|u_j|$, we infer from the Riesz rearrangement properties \citep[Theorem 3.7]{E01} that
\begin{align}
\|u_j^{*}\|_{L^1(\mathbb{R}^d)}= \|u_j\|_{L^1(\mathbb{R}^d)},\quad\|u_j^{*}\|_{L^m(\mathbb{R}^d)}= \|u_j\|_{L^m(\mathbb{R}^d)},\quad\omega(u_j^*)\geq\omega(u_j).
\end{align}
 Thus
\begin{align}\label{C*10}
J(u_j^{*}) \geq J(u_j).
\end{align}
Consequently, $\left(u_j^{*}\right)_j$ is also a maximising sequence and the first step is proved.

We now establish the existence of the supremum. The monotonicity and the non-negativity of $u_j$ imply that
\begin{align}\label{C*11}
\|u_j\|_{L^m(\mathbb{R}^d)}^m=d\alpha_d\int_0^{\infty} r^{d-1}u_j(r)dr\geq \alpha_d R^d\left(u_j(R)\right)^m,
\end{align}
and
\begin{align}\label{C*111}
\|u_j\|_{L^1(\mathbb{R}^d)}\geq \alpha_d R^d\left(u_j(R)\right).
\end{align}
So that
\begin{align}\label{C*12}
0\leq u_j(R)\leq G(R):=C_2 \text{inf}\{R^{-d/m};~R^{-d} \}
\end{align}

 Now, we use once more the monotonicity of the $u_j$’s and their boundedness in $(R, \infty)$ for any $R > 0$ to deduce from Helly's selection principle \cite[P89]{E01} that there are a sub-sequence of $(u _j)_j$ (not relabelled) and a non-negative and non-increasing function $U$ such that $(u_j)_j$ converges to $U$ point-wisely. In addition, as $1<2d/(d+2s)<m$, $\|G(|x|)\|_{L^{2d/(d+2s)}(\mathbb{R}^d)}<\infty$ while the HLS inequality \eqref{HLS1} warrants that $\omega(G(|x|))<\infty$. Together with \eqref{C*12} and the point-wise convergence of $(u_j)_j$, this implies that
\begin{align}\label{C*14}
\lim\limits_{i\to \infty}\omega(u_j)=\omega(U)
\end{align}
by the Lebesgue dominated convergence theorem. In addition, the point-wise convergence of $(u_j)_j$ and Fatou’s lemma ensure
\begin{align}
\|U\|_{L^1(\mathbb{R}^d)}\leq \|u_j\|_{L^1(\mathbb{R}^d)} \text{,} \quad\|U\|_{L^m(\mathbb{R}^d)}\leq \|u_j\|_{L^m(\mathbb{R}^d)}.
\end{align}
Thus
\begin{align}
J(U)=\frac{\omega(U)}{\|U\|_{L^1(\mathbb{R}^{d})}^{2s / d}\|U\|_{L^m(\mathbb{R}^{d})}^{m}}\geq \lim\limits_{j\to \infty}\frac{\omega(u_j)}{\|u_j\|_{L^1(\mathbb{R}^{d})}^{2s / d}\|u_j\|_{L^m(\mathbb{R}^{d})}^{m}} =\lim\limits_{j\to \infty}\omega(u_j)=C_*
\end{align}
and using \eqref{VHLS1} we conclude that
\begin{align}\label{UC}
    J(U)=C_*.
\end{align}
Finally, let us point out that $J(U)=J(\tilde{U})$ where $\tilde{U}=\lambda U(\mu |x-x_0|))$ for $\forall \lambda,~\mu>0$. This completes the proof.
\end{proof}

We are now in a position to begin the study of the identification of the maximisers. To this end, let us define the critical mass $M_{*}$ by
\begin{align}\label{M_*}
M_{*}:=\left[\frac{2}{(m-1) C_{*} c_{d,s}}\right]^{d / 2s} .
\end{align}
and denote
\begin{align}\label{M_*1}
{H}_{M_*}:=\left\{u \in {L}^{1}(\mathbb{R}^{d}) \cap {L}^{m}(\mathbb{R}^{d}):\|u\|_{L^1(\mathbb{R}^d)}=M_*\right\}.
\end{align}

\begin{proposition}[Euler-Lagrange equation]\label{C*}
Let $U$ be a maximiser of $J$ in $H_{M_*}$. Then $U$ is a radially symmetric, non-increasing and compactly supported stationary solution to \eqref{PKS} solving the following Euler-Lagrange equation
\begin{align}\label{EL2}
        \frac{m}{m-1}U^{m-1}-c_{d,s}\int_{\mathbb{R}^d}\frac{U(y)}{|x-y|^{d-2s}}dy=\frac{1}{M_*}\frac{2s}{2s-d}\|U\|_{L^m(\mathbb{R}^d)}^{m},~\text{a.e. in}~ supp(U).
    \end{align}
\end{proposition}
\begin{proof}
Let us describe the set of maximisers of $J$ in $H_{M_*}$. Let us define $\Omega=supp(U)$. Let $\forall \psi\in C_0^\infty(\Omega)$, we define
   \begin{align}\label{EL4}
        \varphi(x)=\psi(x)-\frac{U}{M_*}\int_{\Omega}\psi(x)dx,
    \end{align}
then $supp(\varphi) \subseteq \Omega$ and $\int_{\Omega} \varphi d x=0$. Besides, there exists
\begin{align}\label{EL5}
\varepsilon_{0}:=\frac{\min\limits_{y \in supp \varphi} U(y)}{\max\limits_{y \in supp \varphi}|\varphi(y)|}>0
\end{align}
such that $U+\varepsilon \varphi \geq 0$ in $\Omega$ for $0<\varepsilon<\varepsilon_{0}$. Then a tedious computation shows that
\begin{align}\label{EL6}
&\left.\frac{d}{d \varepsilon}\right|_{\varepsilon=0} J\left(U+\varepsilon \varphi\right)\nonumber\\
=&-\frac{2}{c_{d,s}\|U\|_{L^1(\mathbb{R}^d)}^{2s/d}\|U\|_{L^m(\mathbb{R}^d)}^m}\int_{\mathbb{R}^d}\left(\frac{m}{m-1}U^{m-1}-c_{d,s}\int_{\mathbb{R}^d}\frac{U(y)}{|x-y|^{d-2s}}dy\right)\varphi(x)dx\nonumber\\
&-\frac{\frac{2s}{d}\int_{\mathbb{R}^d}\varphi dx}{\|U\|_{L^1(\mathbb{R}^d)}^{2s/d+1}\|U\|_{L^m(\mathbb{R}^d)}^m}\omega(U)=0
\end{align}
Then by the definition of $\varphi$, \eqref{EL6} can be read as
\begin{align}\label{EL8}
\int_{\Omega}\left(\mathcal{M}-\frac{1}{M_*} \int_{\Omega} \mathcal{M} U d x\right) \psi(x) d x=0, \text { for any } \psi \in C_{0}^{\infty}(\Omega).
\end{align}
Here $\mathcal{M}$ is defined as
\begin{align}\label{EL9}
\mathcal{M}=\frac{m}{m-1} U^{m-1}-c_{d,s}\int_{\mathbb{R}^d}\frac{U(y)}{|x-y|^{d-2s}}dy.
\end{align}
Therefore we discover
\begin{align}\label{EL10}
\frac{m}{m-1} U^{m-1}-c_{d,s}\int_{\mathbb{R}^d}\frac{U(y)}{|x-y|^{d-2s}}dy=\frac{1}{M_*} \left(\int_{\Omega} \frac{m}{m-1} U^{m} d x-c_{d,s}\omega(U)\right), \text { a.e. in } \Omega \text {. }
\end{align}
Taking $J(U)=C_*$ which is stated in \eqref{UC} into account, we arrive at
\begin{align}\label{EL11}
   \frac{m}{m-1}U^{m-1} -c_{d,s}\int_{\mathbb{R}^d}\frac{U(y)}{|x-y|^{d-2s}}dy&=\frac{1}{M_*}\left(\frac{m}{m-1}-\frac{2}{m-1}\right)\int_{\mathbb{R}^d}U^mdx\nonumber\\
   &=\frac{1}{M_*}\frac{2s}{2s-d}\|U\|_{L^m\left(\mathbb{R}^d\right)}^{m}, \text { a.e. in } \Omega \text {. }
\end{align}
Recalling Proposition 3.2 of \cite{B24}, we can proceed to claim that $U$ is a stationary solution to \eqref{PKS}. This completes the proof.
    \end{proof}

\begin{remark}[The stationary solution]
     Similar results for stationary solutions can be obtained in \cite{V17}.
\end{remark}

%%%%%%%%%%%%%%%%%%%%%%%%%%%%%%%%%%%%%%%%%%%%%%%%%%%%%%%%%%%%%%%%%%%%%%%%%%%%%%%%%%%%%%%%%%%%%%%%%%%%%%%%%%%%%%%%%%%%%%%%%%%%%%%%%%%%%%%%%%%%%%%%%%%%%%%%%%%%%%%%%%%%%%%%%%%%%%%%%%%%%%%%%%%%%%%%%%%%%%%%%%%%%%%%%%%%%%%%%%%%%%%%%%%%%%%%%%%%%%%%%%%%%%%%%%%%%%%%%%%%%%%%%%%%%%%%%%%%%%%%%%%%%%%%%%%%%%%%%%%%%%%%%%%%%%%%%%%%%%%%%%%%%%%%%%%%%%%%%%%%%%%%%%%%%%%%%%%%%%%%%%%%%%%%%%%%%%%%%%%%%%%%%%%%%%%%%%%%%%%%%%%%%%%%%%%%%%%%%%%%%%%%%%%%%%%%%%%%%%%%%%%%%%%%%%%%%%%%%%%%%%%%%%%%%%%%%%%%%%%%%%%%%%%%%%%%%%%%%%%%%%%%%%%%%%%%%%%%%%%%%%%%%%%%%%%%%%%%%%%%%%%%%%%%%%%%%%%%%%%%%%%%%%%%%%%%%%%%%%%%%%%%%%%%%%%%%%%%%%%%%%%%%%%%%%%%%%%%%%%%%%%%%%%%%%%%%%%%%%%%%%%%%%%%%%%%%%%%%%%%%%%%%%%%%%%%%%%%%%%%%%%%%%%%%%%%%%%%%%%%%%%%%%%%%%%%%%%%%%%%%%%%%%%%%%%%%%%%%%%%%%%%%%%%%%%%%%%%%%%%%%%%%%%%%%%%%%%%%%%%%%%%%%%%%%%%%%%%%%%%%%%%%%%%%%%%%%%%%%%%%%%%%%%%%%%%%%%%%%%%%%%%%%%%%%%%%%%%%%%%%%%%%%%%%%%%%%%%%%%%%%%%%%%%%%%%%%%%%%%%%%%%%%%%%%%%%%%%%%%%%%%%%%%%%%%%%%%%%%%%%%%%%%%%%%%%%%%%%%%%%%%%%%%%%%%%%%%%%%%%%%%%%%%%%%%%%%%%%%%%%%%%%%%%%%%%%%%%%%%%%%%%%%%%%%%%%%%%%%%%%%%%%%%%%%%%%%%%%%%%%%%%%%%%%%%%%%%%%%%%%%%%%%%%%%%%%%%%%%%%%%%%%%%%%%%%%%%%%%%%%%%%%%%%%%%%%%%%%%%%%%%%%%%%%%%%%%%%%%%%%%%%%%%%%%%%%%%%%%%%%%%%%%%%%%%%%%%%%%%%%%%%%%%%%%%%%%%%%%%%%%%%%%%%%%%%%%%%%%%%%%%%%%%%%%%%%%

\section{ Global existence and finite time blow-up}\label{sec3}

\subsection{ Finite time blow-up}
\begin{lemma}[Second moment]\label{second}
Under assumption \eqref{initial}, let $u$ be a weak solution to \eqref{PKS} for $m=2-\frac{2s}{d}$ on $[0, T )$ with initial condition $u_0$ for some $T \in(0,\infty]$. Then
\begin{align}\label{second2}
\frac{dm_2(t)}{dt}&  =2(d-2s)F(u).
\end{align}
\end{lemma}
\begin{proof}
By integrating by parts in \eqref{PKS} and symmetrising the second term, the time derivative of the second moment of the weak solution is endowed with
\begin{align}\label{second3}
\frac{dm_2(t)}{dt}&=2d\int_{\mathbb{R}^d} u^{m}dx+2\int_{\mathbb{R}^d}(x \cdot \nabla \phi)udx\nonumber\\
&=2d\int_{\mathbb{R}^d} u^{m}dx-\iint_{\mathbb{R}^d\times \mathbb{R}^d}\frac{u(y,t)u(x,t)}{|x-y|^{d-2s}}dydx\nonumber\\
 & =2(d-2s)F(u).
\end{align}
giving the desired identity.
\end{proof}

An easy consequence of the previous lemma is the following blow-up result.

\begin{theorem}[Blowing-up solutions]\label{B}
 If $M>M_{*}$, there exists initial data $u_{0}$ under initial condition \eqref{initial} with $\|u_{0}\|_{L^1(\mathbb{R}^d)}=M$ and negative free energy $F(u_{0})$. Moreover, if $u$ denotes a weak solution to \eqref{PKS} on  $[0, T)$ with initial condition  $u_{0}$, then $T<\infty$ and the ${L}^{m}$-norm of $u$ blows up in finite time.
\end{theorem}
\begin{proof}
{\it\textbf{Step 1}} (The negativity of free energy)\quad
    For $U\in L^1(\mathbb{R}^d)\cap L^m(\mathbb{R}^d)$ is given by \eqref{EL2}, we get $\|U\|_{L^1(\mathbb{R}^d)}=M_*$ and $J(U)=C_*$. If $M>M_{*}$, the initial condition $u_0$ is defined by
    \begin{align}\label{B1}
 u_0=\frac{M}{M_*}U
\end{align}
also satisfies initial condition \eqref{initial} and $\|u_0\|_{L^1(\mathbb{R}^d)}=M$. Some computations show that
\begin{align}\label{B2}
 F(u_0)=F\left(\frac{M}{M_*}U\right)&=\frac{1}{m-1}\int_{\mathbb{R}^d}\left(\frac{M}{M_*}U\right)^mdx-\frac{1}{2}\int_{\mathbb{R}^d}\phi\left(\frac{M}{M_*}U\right)dx\nonumber\\
 &=\frac{1}{m-1}\int_{\mathbb{R}^d}U^mdx\left[\left(\frac{M}{M_*}\right)^m-\frac{M}{M_*}\right]<0.
\end{align}
By combining \eqref{free energy}, \eqref{second2} and \eqref{B2}, we obtain
\begin{align}\label{sf10}
 \frac{dm_2(t)}{dt}=2(d-2s)F(u)\leq2(d-2s)F(u_0)<0
\end{align}
Thus there exists a $T > 0$ such that $\lim\limits_{t\to T}m_2(t)=0$.

{\it\textbf{Step 2}} (Finite time blow-up)\quad
It follows from the H\"{o}lder inequality that
\begin{align}\label{sf2}
\int_{\mathbb{R}^d}  u_0dx=\int_{\mathbb{R}^d}  udx&=\int_{|x|<R}  udx+\int_{|x|\geq R}  udx\nonumber\\
&\leq C_3\left(R^d\right)^{\frac{r-1}{r}}\|u\|_{L^r(\mathbb{R}^d)}+\frac{1}{R^2}m_2(t).
\end{align}
Choosing $R=\left(\frac{m_2(t)}{C_3\|u\|_{ L^r(\mathbb{R}^d)}}\right)^{\frac{1}{d(r-1)/r+2}}$ gives
 \begin{align}\label{sf7}
\int_{\mathbb{R}^d}  u_0dx\leq c\|u\|_{L^r(\mathbb{R}^d)}^{{\frac{2}{d(r-1)/r+2}}}m_2(t)^{\frac{d(r-1)/r}{d(r-1)/r+2}},
\end{align}
It follows that
\begin{align}\label{sf8}
\lim\limits_{t\to T}\text{sup}\|u(\cdot,t)\|_{L^r(\mathbb{R}^d)}\geq \frac{\|  u_0\|_{L^1(\mathbb{R}^d)}^{\frac{d(r-1)/r+2}{2}}}{c m_2(t)^{\frac{d(r-1)/r}{2}}}=\infty .
\end{align}
Thus the proof is completed.
\end{proof}

\subsection{Global existence}
\begin{theorem}[Global existence]\label{GE}
Under assumption \eqref{initial}, if $\|u_0\|_{L^1(\mathbb{R}^d)}=M<M_*$, then there exists a weak solution to \eqref{PKS} on $[0, \infty)$.
\end{theorem}
\begin{proof}
By \Cref{local}, there are $T_{w}$ and a weak solution to \eqref{PKS} in $[0, T_{w})$ with initial data $u_{0}$. We then infer from \eqref{VHLS1} and \eqref{M_*} that there exists a $u_0$ with $\|u_0\|_{L^1(\mathbb{R}^d)}=M<M_*$, for all $t \in [0, T_{w})$,
    \begin{align}\label{GE1}
      F(u_0)\geq  F(u) \geq \frac{C_{*}c_{d,s}}{2}\left(M_*^{\frac{2s}{d}}  -M^{\frac{2s}{d}}\right)\|u\|_{L^m(\mathbb{R}^d)}^m.
    \end{align}
    Therefore, we can deduce that  $u \in L^{\infty}
    \left(0,~ \min (T, T_{w}) ; ~L^{m}(\mathbb{R}^{d})\right)$ for every  $T>0$ which implies that $T_{w}=\infty$  by \Cref{local}.
\end{proof}

\end{document}